\documentclass[a4paper,10pt]{article}
\usepackage[latin1]{inputenc}

\usepackage{lipsum}
\usepackage[dvipsnames]{xcolor}
\usepackage{amsmath}
\usepackage{amssymb}
\usepackage{bm}
\usepackage{caption}
\usepackage{subcaption}
\usepackage{float}
\usepackage[section]{placeins}
\usepackage{amsfonts}
\usepackage{graphicx}
\usepackage{amsthm}
\usepackage{wrapfig}
\usepackage{latexsym}
\newtheorem{theo}{Theorem}

\newtheorem{claim}{Claim}
\newtheorem{lemma}{Lemma}

\newtheorem{cor}{Corollary}
 
\title{Packing mixed hyperarborescences} 

\author{Zolt\'an Szigeti\\  University Grenoble Alpes, Grenoble INP, CNRS,\\ Laboratory G-SCOP,\\ Grenoble, France}

\begin{document}

\maketitle

\begin{abstract}
The aim of this paper is twofold. We first provide a new orientation theorem which gives a natural and simple proof of a result of Gao, Yang \cite{GY} on matroid-reachability-based  packing of  mixed arborescences in mixed graphs by reducing it  to the corresponding theorem of Cs. Kir\'aly \cite{cskir} on directed graphs. Moreover, we extend another result of Gao, Yang \cite{GY2} by providing a new theorem on mixed hypergraphs having a packing of mixed hyperarborescences such that their number is at least $\ell$ and at most $\ell'$, each vertex belongs to exactly $k$ of them, and each vertex $v$ is the root of least $f(v)$ and at most $g(v)$  of them.
\end{abstract}

{\bf Keywords: arborescence, mixed hypergraph, packing}

\section{Introduction} 

This paper is not a survey on packing arborescences. Such a survey is in preparation, see \cite{survey}. We only present here those theorems of the topic that are closely related to the new results of this paper. A preliminary version of the paper appeared in \cite{hujaproc}.

Edmonds \cite{Egy} characterized digraphs having a packing of spanning arborescences with fixed roots. Frank \cite{FA} extended it for a packing of spanning arborescences whose roots are not fixed. The result of Frank \cite{FA}, and independently Cai \cite{cai1}, answers the question when a digraph has an $(f,g)$-bounded packing of spanning arborescences, that is when each vertex $v$ can be the root of at least $f(v)$ and at most $g(v)$ arborescences in the packing. B\'erczi, Frank \cite{BF3} extends it for an $(f,g)$-bounded $k$-regular $(\ell,\ell')$-limited packing of not necessarily spanning arborescences, where $k$-regular means that each vertex  belongs to exactly $k$ arborescences in the packing and $(\ell,\ell')$-limited means that the number of arborescences in the packing is at least $\ell$ and at most $\ell'.$ Kamiyama, Katoh, Takizawa \cite{japan} provided a different type of generalization of Edmonds' theorem in which they wanted to pack reachability arborescences in a digraph, that is each arborescence in the packing must contain all the vertices that can be reached from its root in the digraph. Durand de Gevigney, Nguyen,  Szigeti~\cite{Sz} gave a generalization of Edmonds' theorem where a matroid constraint was added for the packing. More precisely, given a matroid {\sf M} on a multiset of vertices of a digraph $D$, we wanted to have a matroid-based packing of arborescences, that is for every vertex $v$ of $D$, the set of roots of the arborescences in the packing containing $v$ must form a basis of {\sf M}. In \cite{cskir} Cs. Kir\'aly  proposed a common generalization of the previous two results. He characterized pairs $(D,{\sf M})$ of a digraph and a matroid  that have a matroid-reachability-based packing of arborescences, that is for every vertex $v$ of $D$, the set of roots of the arborescences in the packing containing $v$  must form a basis of the subset of the elements of {\sf M} from which $v$ is reachable in $D.$

All of these results hold for dypergraphs, see \cite{fkiki}, \cite{HSz}, \cite{survey}, \cite{BF2}, \cite{FKLSzT}, and for mixed graphs, see \cite{FA}, \cite{GY}, \cite{survey}, \cite{MT}, \cite{FKLSzT}, \cite{GY2}. In fact, all of these results, except the one of B\'erczi, Frank \cite{BF3}, are known to hold for mixed hypergraphs, see \cite{FKLSzT}, \cite{HSz}, \cite{HSz2}. The present paper will fill in this gap by showing that this result also holds for mixed hypergraphs. More precisely, we will characterize mixed hypergraphs having an $(f,g)$-bounded $k$-regular $(\ell,\ell')$-limited packing of    mixed hyperarborescences. Our theorem naturally  generalizes a result of Gao, Yang \cite{GY2} on $(f,g)$-bounded  packing of $k$  spanning mixed arborescences and will follow from the theory of generalized polymatroids. The other aim of this paper is to provide a new proof of another result of Gao, Yang \cite{GY} on matroid-reachability-based  packing of mixed arborescences. Our approach is to reduce the result to the result of Cs. Kir\'aly \cite{cskir} on matroid-reachability-based  packing of  arborescences via a new orientation theorem.

The organization of the paper is as follows. In Section \ref{arbor} we consider problems related to matroid-reachability-based  packings of  mixed arborescences. In Section \ref{hyperarbor} we consider problems related to $(f,g)$-bounded $k$-regular $(\ell,\ell')$-limited packings of mixed hyperarborescences.

\section{Definitions}
A {\it multiset} of $V$ may contain multiple occurrences of elements. For a multiset $S$ of $V$ and a subset $X$ of $V$, {\boldmath$S_X$} denotes the multiset consisting of the elements of $X$ with the same multiplicities as in $S$ and  {\boldmath$\overline{X}$} denotes $V-X.$ A set of disjoint subsets of $V$ is called a {\it subpartition} of $V.$ For a subpartition ${\cal P}$ of $V$, {\boldmath$\cup{\cal P}$} denotes the set of elements  that belong to some member of ${\cal P}$. A subpartition ${\cal P}$ of $V$ is a {\it partition} of $V$ if $\cup{\cal P}=V.$ For a function $h$ on $V$ and a subset $X$ of $V,$ $h(X)=\sum_{v\in X}h(v).$

A set function $p$ ($b$) on $S$ is called {\it supermodular} ({\it submodular}) if \eqref{supermod} (\eqref{submod}, respectively) holds. The in-degree function of a digraph and the rank function of a matroid are well-known examples of submodular set functions. 
 If $p$ satisfies the supermodular  inequality  for all intersecting sets for then $p$ is called {\it intersecting supermodular}.
\begin{eqnarray}
 p(X)+p(Y)&\le &p(X\cap Y)+p(X\cup Y) \text{ for all } X,Y\subseteq S,\label{supermod}\\ 
 b(X)+b(Y)&\ge &b(X\cap Y)+b(X\cup Y) \ \text{ for all }  X,Y\subseteq S.\label{submod}
 \end{eqnarray} 

Let {\boldmath$D$} $=(V,A)$ be a directed graph, shortly {\it digraph}. For a  subset $X$ of $V,$ the set of arcs in $A$ {\it entering $X$} is denoted by {\boldmath$\rho_A(X)$}  and the {\it in-degree} of $X$ is {\boldmath$d^-_A(X)$} $=|\rho_A(X)|.$ For a subset  $X$ of $V,$ we denote by {\boldmath$P_D^X$} ({\boldmath$Q_D^X$}) the set of vertices from (to) which there exists a path to (from, respectively) at least one vertex of $X$. We say that $D$ is an {\it arborescence with root $s$}, shortly {\it $s$-arborescence}, if $s\in V$ and there exists a unique path from $s$ to $v$ for every $v\in V;$ or equivalently, if $D$ contains no circuit and every vertex in $V-s$ has in-degree $1.$ We say that $D$ is a {\it branching with root set $S$} if $S\subseteq V$ and there exists a unique path from $S$ to $v$ for every $v\in V.$  A subgraph of $D$ is called {\it spanning}  if its vertex set is $V.$
 A subgraph of $D$ is called a {\it reachability $s$-arborescence} if it is an $s$-arborescence and its vertex set is $Q_D^s.$
By a {\it packing} of subgraphs in $D$, we mean a set of subgraphs that are arc-disjoint. A packing of subgraphs is called {\it $k$-regular} if every vertex belongs to exactly $k$ subgraphs in the packing. For two functions $f,g: V\rightarrow \mathbb Z_+$, a packing of arborescences is called {\it $(f,g)$-bounded} if  the number of $v$-arborescences in the packing is at least $f(v)$ and at most $g(v)$ for every $v\in V$. For $\ell,\ell'\in\mathbb Z_+,$ a packing of arborescences is called {\it $(\ell,\ell')$-limited} if the number of arborescences in the packing is at least $\ell$ and at most $\ell'.$ For  a multiset $S$ of $V$ and  a matroid ${\sf M}$ on $S,$ a packing of arborescences in $D$  is called {\it  matroid-based} (resp. {\it  matroid-reachability-based}) if every $s\in S$ is the root of at most one arborescence in the packing and for every  $v\in V$, the multiset of roots  of arborescences containing $v$ in the packing forms a basis of $S$ ($S_{P_D^v}$, respectively) in ${\sf M}.$

Let {\boldmath$F$} $=(V,E\cup A)$ be a {\it mixed graph}, where $E$ is a set of {edges} and $A$ is a set of {arcs}. A mixed subgraph $F'$ of $F$ is a {\it mixed path} if the edges in $F'$ can be oriented in such a way that we obtain a directed path.
For a subset  $X$ of $V,$ we denote by {\boldmath$P_F^X$} ({\boldmath$Q_F^X$}) the set of vertices from (to) which there exists a mixed path to (from, respectively) at least one vertex of $X$. We say that $F$ is {\it strongly connected} if there exists a mixed path from $s$ to $t$ for all $(s,t)\in V^2.$ The maximal strongly connected subgraphs of $F$ are called the {\it strongly connected components} of $F$.  A {\it mixed $s$-arborescence} is a mixed graph  that has an orientation that is an $s$-arborescence. 
 A mixed subgraph of $F$ is called a {\it spanning (reachability) mixed $s$-arborescence} if it is a mixed $s$-arborescence and its vertex set is $V$ ($Q_F^s$, respectively). By a {\it packing} of subgraphs in $F$, we mean a set of subgraphs that are edge- and arc-disjoint. All the packing problems considered in digraphs can also be considered in mixed graphs. 

 Let {\boldmath$\mathcal{D}$} $=(V,\mathcal{A})$ be a directed hypergraph, shortly {\it dypergraph}, where {\boldmath$\mathcal{A}$} is the set of dyperedges of $\mathcal{D}.$ A {\it dyperedge} $e$ is an ordered pair $(Z,z)$, where $z\in V$ is the  {\it head}  and $\emptyset\neq Z\subseteq V-z$ is the set of {\it tails}  of $e.$ For $X\subseteq V,$ a dyperedge $(Z,z)$ {\it enters $X$} if $z\in X$ and $Z\cap \overline X\neq\emptyset.$ The  set of dyperedges in $\mathcal{A}$ {\it entering $X$} is denoted by {\boldmath$\rho_\mathcal{A}(X)$}  and the {\it in-degree} of $X$ is {\boldmath$d^-_\mathcal{A}(X)$} $=|\rho_\mathcal{A}(X)|.$ 
By {\it trimming} a dyperedge $(Z,z)$, we mean the operation that replaces $(Z,z)$ by an arc $yz$ where $y\in Z.$  We say that $\mathcal{D}$ is  a {\it hyperarborescence with root $s$}, shortly {\it $s$-hyperarborescence}, if $\mathcal{D}$ can be trimmed to an $s$-arborescence. We mention that  we  delete vertices that became isolated vertices during the trimming, that is the vertex set of the arborescence is not necessarily the vertex set of $\mathcal{D}$.  
 We say that $\mathcal{D}$ is a {\it hyperbranching with root set $S$} if $\mathcal{D}$ can be trimmed to a branching with root set $S$ (the resulting isolated vertices in $V-S$ are deleted). If $S=\{s\}$ then a hyperbranching with root set $S$ is an $s$-hyperarborescence.
A hyperbranching $(V,\mathcal{A}')$ with root set $S$ is called {\it spanning} in $\mathcal{D}$ if $\mathcal{A}'\subseteq \mathcal{A}$ and  $|\mathcal{A}'|=|V|-|S|.$ 
A {\it packing} of subdypergraphs in $\mathcal{D}$ is a set of subdypergraphs that are dyperedge-disjoint.
We say that $\mathcal{D}$ has a {\it matroid-based}/{\it $(f,g)$-bounded}/{\it $k$-regular}/{\it $(\ell,\ell')$-limited} packing of hyperarborescences if $\mathcal{D}$ can be trimmed to a digraph $(V,A)$ that has  a {matroid-based}/{$(f,g)$-bounded}/{$k$-regular}/{$(\ell,\ell')$-limited} packing of arborescences.

Let {\boldmath$\mathcal{F}$} $=(V,\mathcal{E}\cup \mathcal{A})$ be a {\it mixed hypergraph}, where {\boldmath$\mathcal{E}$} is the set of {hyperedges} and {\boldmath$\mathcal{A}$} is the set of {dyperedges} of $\mathcal{F}$. A {\it hyperedge} is a subset of $V$ of size at least two. A hyperedge $X$ {\it enters} a subset $Y$ of $V$ if $X\cap Y\neq\emptyset\neq \overline X\cap Y.$
By {\it orienting} a hyperedge $X,$ we mean the operation that replaces the hyperedge $X$ by a dyperedge $(X-x,x)$ for some $x\in X$. 
For $\vec{\mathcal{Z}}\subseteq \mathcal{A}$,    $\mathcal{Z}$ denotes the set of underlying hyperedges of $\vec{\mathcal{Z}}$. 
For  ${\mathcal{Z}}\subseteq \mathcal{E}$ and $X\subseteq V$, we denote by {\boldmath$V(\mathcal{Z})$} the set of vertices that belong to at least one hyperedge in $\mathcal{Z}$ and by {\boldmath$\mathcal{Z}(X)$} the set of hyperedges in  $\mathcal{Z}$ that are contained in $X.$
 A {\it mixed $s$-hyperarborescence} is a mixed hypergraph that has an orientation that is an $s$-hyperarborescence. 
A mixed $s$-hyperarborescence $\mathcal{B}=(V,\mathcal{E}'\cup \mathcal{A}')$ is called {\it spanning} in $\mathcal{F}$ if $\mathcal{E}'\subseteq \mathcal{E}$, $\mathcal{A}'\subseteq \mathcal{A}$, and  $|\mathcal{E}'|+|\mathcal{A}'|=|V|-1.$ 
For a family ${\cal P}$ of subsets of $V$, we denote by {\boldmath$e_{\mathcal{E}\cup \mathcal{A}}({\cal P})$} the number of hyperedges in $\mathcal{E}$ and dyperedges in $\mathcal{A}$ that enter some member of ${\cal P}$. For $X\subseteq V$, we use $e_{\mathcal{E}\cup \mathcal{A}}(X)$ for $e_{\mathcal{E}\cup \mathcal{A}}(\{X\}).$
A {\it packing} of mixed subhypergraphs in $\mathcal{F}$ is a set of mixed subhypergraphs that are hyperedge- and dyperedge-disjoint.
We say that $\mathcal{F}$ has an {\it $(f,g)$-bounded}/{\it $k$-regular}/{\it $(\ell,\ell')$-limited} packing of mixed hyperarborescences if $\mathcal{E}$ has an orientation $\vec{\mathcal{E}}$ such that the dypergraph $(V,\vec{\mathcal{E}}\cup\mathcal{A})$  has  an {$(f,g)$-bounded}/{$k$-regular}/{$(\ell,\ell')$-limited} packing of hyperarborescences. 

\section{Packing mixed arborescences}\label{arbor}

In this section we list known results on packing mixed arborescences that are related to our first contribution. We propose a new approach to prove a result of Gao, Yang \cite{GY} on matroid-reachability-based  packing of  mixed arborescences via a new orientation theorem,  and we provide its proof.
\medskip

We start with the fundamental result of Edmonds \cite{Egy} on packing spanning arborescences with fixed roots.
\begin{theo}[Edmonds \cite{Egy}]\label{edmondsarborescences} 
Let $D=(V,A)$ be a digraph and $S$ a multiset of  $V.$
There exists a packing of spanning $s$-arborescences $(s\in S)$ in $D$ if and only if 
	\begin{eqnarray*} 
		 d^-_A(X) \geq |S|-|S_X|  \text{ for every  $\emptyset\neq X\subseteq V.$}
	\end{eqnarray*}
\end{theo}

Frank \cite{FA} extended Theorem \ref{edmondsarborescences} for mixed graphs. Here we  present a seemingly more general version of it but it is equivalent to the original result.

\begin{theo}[Frank \cite{FA}]\label{mixedgraph}
Let $F=(V,E\cup A)$ be a mixed graph,   $S$ a multiset of  $V.$
There exists a packing of spanning  mixed  $s$-arborescences $(s\in S)$ in  $F$ if and only if 
\begin{eqnarray} \label{condmixed}
	e_{E\cup A}({\cal P}) &\geq& |S||{\cal P}|-|S_{\cup\mathcal{P}}| \hskip .44truecm \text { for every subpartition } {\cal P} \text { of } V.
\end{eqnarray}
\end{theo}

An elegant extension of Theorem \ref{edmondsarborescences} for packing reachability arborescences was provided in \cite{japan}.

\begin{theo}[Kamiyama, Katoh, Takizawa \cite{japan}]\label{reach1} 
Let $D=(V,A)$ be a digraph and $S$ a multiset of  $V.$
There exists a packing of reachability $s$-arborescences $(s\in S)$ in $D$ if and only if 
	\begin{eqnarray*}  
		d^-_A(X) \geq |S_{P_D^X}|-|S_X|  \text{ for every  $X\subseteq V.$}
	\end{eqnarray*}
\end{theo}

When each vertex is reachable from every vertex of $S$, Theorem \ref{reach1} reduces to Theorem \ref{edmondsarborescences}.
Theorem \ref{reach1} can be proved by induction and using Edmonds' result on packing  spanning branchings, see H\"orsch, Szigeti \cite{HSz2}.
\medskip

 Theorem \ref{reach1} can also be generalized for mixed graphs follows. For convenience, we present not the original version of the result but one due to Gao, Yang \cite{GY} that fits better to our framework. 

\begin{theo}[Matsuoka, Tanigawa \cite{MT}]\label{mattan}
Let $F=(V,E\cup A)$ be a mixed graph and $S$ a multiset of  $V.$
There exists a packing of reachability mixed $s$-arborescences $(s\in S)$ in $F$ if and only if 
for every strongly connected component $C$ of $F$ and every  set ${\cal P}$ of subsets of $P_F^C$ such that $Z\cap C\neq\emptyset$ and $e_{E\cup A}(Z-C)=0$ for every $Z\in{\cal P}$ and $Z\cap Z'\cap C=\emptyset$ for every $Z,Z'\in{\cal P}$,
\begin{eqnarray*}
e_{E\cup A}({\cal P})\geq \sum_{Z\in {\cal P}}(|S_{P_F^C}|-|S_Z|).
\end{eqnarray*}
\end{theo}

If $F$ is a digraph then Theorem \ref{mattan} reduces to Theorem \ref{reach1}. 
\medskip

Another type of generalizations of Theorem \ref{edmondsarborescences} was obtained by adding a matroid constraint.

\begin{theo}[Durand de Gevigney, Nguyen,  Szigeti~\cite{Sz}] \label{thmSz}
Let $D=(V,A)$ be a  digraph, $S$ a multiset of  $V$ and ${\sf M}=(S,r_{\sf M})$ a matroid. There exists a  ${\sf M}$-based packing of arborescences in $D$  if and only if 
	\begin{eqnarray*}
		d^-_{A}(X)	\geq	 r_{{\sf M}}(S)-r_{{\sf M}}(S_X) \text{ for every } \emptyset\neq X\subseteq V. %\label{matcondori2}
	\end{eqnarray*}
\end{theo}

For the free matroid {\sf M}, Theorem \ref{thmSz} reduces to Theorem \ref{edmondsarborescences}.
\medskip

A common generalization of Theorems \ref{reach1} and \ref{thmSz} was found by Cs. Kir\'aly \cite{cskir}.

\begin{theo}[Cs. Kir\'aly \cite{cskir}] \label{thmCsaba}
Let $D=(V,A)$ be a  digraph, $S$ a multiset of  $V$ and ${\sf M}=(S,r_{\sf M})$ a matroid. 
There exists a  matroid-reachability-based packing of arborescences in $D$  if and only if  
	\begin{eqnarray}\label{csabicond}
		d^-_{A}(X)&\geq &r_{{\sf M}}(S_{P_D^X})-r_{{\sf M}}(S_X) \hskip .44truecm\text{ for every } X\subseteq V.
	\end{eqnarray}
\end{theo}

For the free matroid {\sf M}, Theorem \ref{thmCsaba} reduces to Theorem \ref{reach1}. When each vertex is reachable from a basis of {\sf M}, Theorem \ref{thmCsaba} reduces to Theorem \ref{thmSz}.
\medskip

Gao, Yang \cite{GY} provided another characterization of the existence of a matroid-reachability-based  packing of   arborescences.

\begin{theo}[Gao, Yang \cite{GY}]\label{lblbkjlkdir}
Let $D=(V,A)$ be a digraph, $S$ a multiset of  $V$ and ${\sf M}=(S,r_{\sf M})$ a matroid.
There exists a  matroid-reachability-based  packing of   arborescences  in $D$ if and only if  for every strongly connected component $C$ of $D$  
and every $X\subseteq P_D^C$ such that $X\cap C\neq\emptyset$ and $d^-_{A}(X-C)=0$,
\begin{eqnarray}\label{lknemlkbfljevdhuedir}
d_A^-(X)&\geq& r_{\sf M}(S_{P_D^C})-r_{\sf M}(S_X).
\end{eqnarray}
\end{theo}

Let us show that Theorems \ref{thmCsaba} and \ref{lblbkjlkdir} are equivalent. 

\begin{proof}
We have to prove that \eqref{csabicond} and \eqref{lknemlkbfljevdhuedir} are equivalent.
\medskip

\noindent \eqref{csabicond} {\boldmath $\Longrightarrow$} \eqref{lknemlkbfljevdhuedir}: If  \eqref{csabicond} holds then let  $C$ be  a strongly connected component of $D$ and $X\subseteq P_D^C$ such that $X\cap C\neq\emptyset$ and $d^-_{A}(X-C)=0$. 
Then, we have  $P_D^{X}=P_D^C$ 
and hence \eqref{csabicond} implies \eqref{lknemlkbfljevdhuedir}. 
\medskip

\noindent \eqref{lknemlkbfljevdhuedir} {\boldmath $\Longrightarrow$} \eqref{csabicond}: Now if \eqref{lknemlkbfljevdhuedir} holds then let $X$ be a subset of $V.$ Let $C_1,\dots, C_k$ be the strongly connected components of $D$ in a topological ordering that is if there exists an arc from $C_i$ to $C_j$ then $i<j$. 
Let 
\begin{eqnarray*}
\text{\boldmath$J$} &=&\{1\le j\le k:X\cap C_j\neq\emptyset\}, \\
\text{\boldmath${X}_j$} &=&(X\cap C_j)\cup\bigcup\limits_{\substack{i\in J-\{j\} \\ C_i\subseteq P_D^{C_j}}}P_D^{C_i} \hskip 1truecm \text{ for every } j\in J. 
\end{eqnarray*}
Note that $X_j\subseteq P_D^{C_j}$, $X_j\cap C_j\neq\emptyset$ and $d^-_{A}(X_j-C_j)=0$ for every $j\in J.$

\begin{claim}\label{axaxi}
$d_A^-(X)\ge\sum_{j\in J}d_A^-(X_j)$. 
\end{claim}

\begin{proof}
If $uv$ enters $X_j$ then $v\in X\cap C_j\subseteq X$ and $u\notin X_j.$ If $u\in X$ then $u\in X\cap C_{j'}$ for some $j'\in J$.  Since $C_{j'}$ is strongly connected, $u\in C_{j'}$ and $v\in X\cap C_j$, we have $C_{j'}\subseteq P_D^{C_j},$ so $u\in X_j$ which is a contradiction. It follows that $u\notin X,$ so $uv$ enters $X.$ Since $(X\cap C_j)\cap (X\cap C_{j'})=\emptyset$ for distinct $j,j'\in J,$ the claim follows.
\end{proof}

\begin{claim}\label{mspd}
$\sum_{j\in J}(r_{\sf M}(S_{P_D^{C_j}})-r_{\sf M}(S_{X_j}))\ge r_{\sf M}(S_{P_D^X})-r_{\sf M}(S_X).$ 
\end{claim}

\begin{proof}
We prove it by induction on $|J|.$ For $|J|=1$, say $J=\{j\},$ the claim follows from $P_D^{C_j}=P_D^{X_j}.$ Suppose that the inequality holds for $|J|-1$. Let $\ell$ be the largest value in $J.$ 
Note that we have
\begin{eqnarray*}
P_D^{X_\ell}\cap (X_\ell\cup P_D^{X-X_\ell})&\supseteq &X_\ell, \hskip 1truecm
P_D^{X-C_\ell}\cap X\ \ \supseteq \ \ X-C_\ell,\\
P_D^{X_\ell}\cup (X_\ell\cup P_D^{X-X_\ell})&\supseteq &P_D^X, \hskip .9truecm
P_D^{X-C_\ell}\cup X\ \ \supseteq \ \ X_\ell\cup P_D^{X-C_\ell}.
\end{eqnarray*}
Then, by induction, submodularity of $r_{\sf M}$, first for $S_{P_D^{X_\ell}}$ and $S_{X_\ell\cup P_D^{X-X_\ell}}$, then for $S_{P_D^{X-C_\ell}}$ and $S_X,$ and monotonicity of $r_{\sf M}$, we have 
\begin{eqnarray*}
\sum_{j\in J}(r_{\sf M}(S_{P_D^{C_j}})-r_{\sf M}(S_{X_j}))	 &	\ge 	&	(r_{\sf M}(S_{P_D^{X_\ell}})-r_{\sf M}(S_{X_\ell}))   +   (r_{\sf M}(S_{P_D^{X-C_\ell}})-r_{\sf M}(S_{X-C_\ell}))\\
&	\ge 	&	(r_{\sf M}(S_{P_D^X})-r_{\sf M}(S_{X_\ell\cup P_D^{X-X_\ell}}))  + (r_{\sf M}(S_{X_\ell\cup P_D^{X-C_\ell}})-r_{\sf M}(S_X))\\
&	\ge	&	r_{\sf M}(S_{P_D^X})-r_{\sf M}(S_X),
\end{eqnarray*}
and the claim follows.
\end{proof}

By Claim \ref{axaxi},  \eqref{lknemlkbfljevdhuedir} applied for all $X_j,$ and Claim \ref{mspd}, we get that 
\begin{eqnarray*}
d_A^-(X)\ge\sum_{j\in J}d_A^-(X_j)\ge\sum_{j\in J}(r_{\sf M}(S_{P_D^{C_j}})-r_{\sf M}(S_{X_j}))\ge r_{\sf M}(S_{P_D^X})-r_{\sf M}(S_X),
\end{eqnarray*}
so \eqref{csabicond} holds. 
\end{proof}

A common generalization of Theorems \ref{mattan} and \ref{lblbkjlkdir} was provided by Gao, Yang \cite{GY}.

\begin{theo}[Gao, Yang \cite{GY}]\label{lblbkjlk}
Let $F=(V,E\cup A)$ be a mixed graph, $S$ a multiset of  $V$ and ${\sf M}=(S,r_{\sf M})$ a matroid.
There exists a  matroid-reachability-based  packing of  mixed arborescences  in $F$ if and only if  for every strongly connected component $C$ of $F$ and every  set ${\cal P}$ of subsets of $P_F^C$ such that $Z\cap C\neq\emptyset$ and $e_{E\cup A}(Z-C)=0$ for every $Z\in{\cal P}$ and $Z\cap Z'\cap C=\emptyset$ for every $Z,Z'\in{\cal P}$,
\begin{eqnarray}\label{lknemlkbfljevdhue}
e_{E\cup A}({\cal P})&\geq&\sum_{Z\in {\cal P}}(r_{\sf M}(S_{P_F^C})-r_{\sf M}(S_Z)).
\end{eqnarray}
\end{theo}

For the {\it free} matroid {\sf M}, that is every set of $S$ is independent in {\sf M}, Theorem \ref{lblbkjlk} reduces to Theorem \ref{mattan}.
For $E=\emptyset,$ Theorem \ref{lblbkjlk} reduces to Theorem \ref{lblbkjlkdir}.
H\"orsch, Szigeti \cite{HSz2}  pointed out that  Theorem \ref{lblbkjlk} holds for mixed hypergraphs.  That more general result was proved in \cite{HSz2} by induction using a result on matroid-based packing of mixed hyperbranchings in mixed hypergraphs from \cite{FKLSzT}.
Here we propose another approach to prove Theorem \ref{lblbkjlk}. It will be derived from its directed version (Theorem \ref{thmCsaba}) using a new orientation result (Theorem \ref{lknvlknlkklsz}). 
To prove   Theorem \ref{lknvlknlkklsz} we  need a result of Frank, see Theorem 15.4.13 in \cite{book}.

\begin{theo}[Frank \cite{book}]  \label{hypgraphcoveringhorient}
Let $G=(V,E)$ be a graph and $h$ an integer-valued intersecting supermodular set function  such that $h(V)=0.$ There exists an  orientation $\vec G=(V,\vec E)$ of  $G$ such that 
	\begin{eqnarray} \label{coverh}
		d^-_{\vec E}(X)&\ge &h(X) \hskip .4truecm\text{ for every } X\subseteq V \hskip 1.4truecm
	 \end{eqnarray}
if and only if
	\begin{eqnarray} \label{condhyph}
		e_{E}({\cal P}) &\geq& \sum_{X\in {\cal P}}h(X) \hskip .4truecm\text{ for every subpartition } {\cal P} \text{ of } V.
	 \end{eqnarray}
\end{theo}

We can now present and prove our first contribution, a new orientation theorem. It will allow us to reduce Theorem \ref{lblbkjlk} to Theorem \ref{thmCsaba}.
The motivation of the use of $h(X)-h(P_F^X)$ in \eqref{hgjgfxhfdwfdsz} is the following. In order to be able to apply Theorem \ref{thmCsaba} we want to find an orientation $\vec F=(V,\vec E\cup A)$ of a mixed graph $F=(V,E\cup A)$ such that \eqref{csabicond} holds in $\vec F$, that is $d^-_{\vec{E}\cup A}(X) \ge r_{{\sf M}}(S_{P_{\vec F}^X})-r_{{\sf M}}(S_X)$ and $r_{{\sf M}}(S_{P_{\vec F}^X})=r_{{\sf M}}(S_{P_F^X})$ for every $X\subseteq V$ or equivalently  $d^-_{\vec{E}\cup A}(X) \ge r_{{\sf M}}(S_{P_F^X})-r_{{\sf M}}(S_X)$ for every $X\subseteq V$ which is \eqref{hgjgfxhfdwfdsz} for $h(X)=-r_{{\sf M}}(S_X)$.

\begin{theo}\label{lknvlknlkklsz}
Let $F=(V,E\cup A)$ be a mixed graph and $h$ an integer-valued intersecting supermodular set function on $V.$
There exists an orientation  $\vec E$ of $E$ such that  	
	\begin{eqnarray}\label{hgjgfxhfdwfdsz}
		d^-_{\vec{E}\cup A}(X)&\geq &h(X)-h(P_F^X) \hskip .44truecm\text{ for every } X\subseteq V
	\end{eqnarray}
if and only if  for every strongly connected component $C$ of $F$ and every set ${\cal P}$ of subsets of $P_F^C$ such that $Z\cap C\neq\emptyset, e_{E\cup A}(Z-C)=0$ for every $Z\in{\cal P};$ and $Z\cap Z'\cap C=\emptyset$ for every $Z,Z'\in{\cal P}$,
\begin{eqnarray}\label{mglvjhdfufehomf}
e_{E\cup A}({\cal P})&\geq& \sum_{Z\in {\cal P}}(h(Z)-h(P_F^C)).\hskip 1.8truecm
\end{eqnarray}
\end{theo}

\begin{proof} 
To prove the {\bf necessity}, let  {\boldmath$\vec E$} be an orientation  of $E$ such that \eqref{hgjgfxhfdwfdsz} holds, {\boldmath$C$} a strongly connected component of $F$ 
 and {\boldmath${\cal P}$} a set of subsets of $P_F^C$ such that $Z\cap C\neq\emptyset, e_{E\cup A}(Z-C)=0$ for every $Z\in{\cal P};$ and $Z\cap Z'\cap C=\emptyset$ for all $Z,Z'\in{\cal P}.$ It follows that $\rho_{\vec{E}\cup A}(Z)\cap\rho_{\vec{E}\cup A}(Z')=\emptyset$  and $P_F^Z=P_F^C$ for all $Z,Z'\in{\cal P}.$ Then, by \eqref{hgjgfxhfdwfdsz} applied for every $Z\in {\cal P}$, we obtain  	\eqref{mglvjhdfufehomf} because 
 \begin{eqnarray}
 		e_{E\cup A}({\cal P})\ge e_{\vec{E}\cup A}({\cal P})=\sum_{Z\in {\cal P}}d_{\vec{E}\cup A}(Z)\ge \sum_{Z\in {\cal P}}(h(Z)-h(P_F^C)).
 	\end{eqnarray}
 
To prove the {\bf sufficiency}, let ({\boldmath $F$} = ({\boldmath $V$}, {\boldmath $E \cup A$}), {\boldmath $h$}) be a  counterexample for Theorem \ref{lknvlknlkklsz} that minimizes $|V|$. Let {\boldmath $C$} be a strongly connected component of $F$ such that  $e_{E\cup A}(\overline C)=0$. Let ({\boldmath $F_1$} = ({\boldmath $V_1$}, {\boldmath $E_1\cup A_1$}), {\boldmath $h_1$}) be 
obtained from $(F,h)$ by deleting the elements in $C.$
As $e_{E\cup A}(\overline C)=0$, we have $e_{E_1\cup A_1}(X)=e_{E\cup A}(X)$, $P_{F_1}^X=P_F^X$ and $h_1(X)=h(X)$ for every $X\subseteq V_1$. Then, since $(F,h)$ satisfies \eqref{mglvjhdfufehomf}, so does $(F_1,h_1).$ Hence, by the minimality of $(F,h)$,  there exists an orientation  {\boldmath$\vec E_1$} of $E_1$ such that 
	\begin{eqnarray}\label{kbxjhsvhchcvkjsz}
		d^-_{\vec{E}_1\cup A_1}(X)&\geq &h(X)-h(P_F^X) \hskip .44truecm\text{ for every } X\subseteq V_1.
	\end{eqnarray}
Let us now consider the subgraph {\boldmath $F_2$} = ({\boldmath $C$}, {\boldmath $E_2\cup A_2$}) of $F$ induced by $C$. Moreover, let us define {\boldmath$h_2$}$(X)=\max\{h(Y)-d_A^-(Y): Y\subseteq P_F^C, Y\cap C=X, e_{E\cup A}(Y-C)=0\}$ for every $\emptyset\neq X\subseteq C.$  For any non-empty set $X_i$ in $C,$ let {\boldmath$Y_i$} be a set that provides $h_2(X_i).$
Gao, Yang \cite{GY} proved that $h_2$ is intersecting supermodular.
\begin{claim}\label{bvesszoszubmod}
 $h_2$ is an intersecting supermodular set function on  $C.$
\end{claim}

\begin{proof}
For intersecting sets $X_1$ and $X_2$ in $C,$ let {\boldmath$X_3$} $=X_1\cap X_2$, {\boldmath$X_4$} $=X_1\cup X_2$, {\boldmath$Y'_3$} $=Y_1\cap Y_2$ and {\boldmath$Y'_4$} $=Y_1\cup Y_2.$ Note that, for $i=3,4,$ we have $Y'_i\subseteq P_F^C,Y'_i\cap C=X_i$ and $e_{E\cup A}(Y'_i-C)=0,$ and hence $h(Y'_i)-d_A^-(Y'_i)\le h_2(X_i).$ Then, by the intersecting supermodularity of $h$ and $-d_A^-$, we get that  
\begin{eqnarray*}
h_2(X_1)+h_2(X_2)		&	=	&	h(Y_1)-d_A^-(Y_1)+h(Y_2)-d_A^-(Y_2)\\ 
					&	\le 	&	h(Y'_3)-d_A^-(Y'_3)+h(Y'_4)-d_A^-(Y'_4)\\
					&	\le 	&	h_2(X_3)+h_2(X_4)\\
					&	 =	&	h_2(X_1\cap X_2)+h_2(X_1\cup X_2),
\end{eqnarray*}
so $h_2$ is intersecting supermodular.
\end{proof}

Let {\boldmath$h'$} be defined by $h'(X)=h_2(X)-h(P_F^C)$ for every $\emptyset\neq X\subseteq C$ and $h'(\emptyset)=0.$ By the Claim \ref{bvesszoszubmod}, $h'$ is intersecting supermodular on $C.$ Let {\boldmath${\cal P}$} $=\{X_1,\dots,X_t\}$ be a subpartition of $C$ 
and {\boldmath${\cal P}'$} $=\{Y_i:X_i\in{\cal P}\}$. Then ${\cal P}'$ is a set of subsets of $P_F^C$ such that $Y_i\cap C\neq\emptyset$ and $e_{E\cup A}(Y_i-C)=0$ for $1\le i\le t$ and $Y_i\cap Y_j\cap C=\emptyset$ for $1\le i<j\le t$. It follows, by \eqref{mglvjhdfufehomf}, that 
\begin{eqnarray*}
e_{E_2}({\cal P})&=&e_{E\cup A}({\cal P}')-e_{A}({\cal P}')\\
&\ge &\sum_{Y_i\in{\cal P}'}(h(Y_i)-h(P_F^C)-d_A^-(Y_i))\\
&= &\sum_{Y_i\in{\cal P}'}(h_2(X_i)-h(P_F^C))\\
&=&\sum_{X_i\in{\cal P}} h'(X_i).
\end{eqnarray*} 
Thus the graph $(C,E_2)$ satisfies \eqref{condhyph}. In particular, we get that $0=e_{E_2}(C)\ge h'(C).$ Moreover, $h'(C)=h_2(C)-h(P_F^C)\ge h(P_F^C)-h(P_F^C)=0.$ Hence $h'(C)=0.$ 
Then, by Theorem \ref{hypgraphcoveringhorient}, there exists an orientation  {\boldmath$\vec {E}_2$} of $E_2$ such that $d_{\vec {E}_2}^-(X)\ge h'(X)=h_2(X)-h(P_F^C)$  for every $X\subseteq C.$ It follows that for every $Y\subseteq P_F^C$ with $Y\cap C\neq\emptyset$ and $e_{E\cup A}(Y-C)=0,$ we have 
	\begin{eqnarray}\label{mlioniuib}
d^-_{\vec {E}_2}(Y)=d^-_{\vec {E}_2}(Y\cap C)\ge h(Y)-h(P_F^C)-d^-_{A}(Y).
	\end{eqnarray}
	
Let {\boldmath$\vec F$} $=(V,\vec E\cup A),$ where $\vec E=\vec {E}_1\cup \vec {E}_2$. To finish the proof we show that $\vec F$ satisfies  \eqref{hgjgfxhfdwfdsz}. If $X\subseteq V_1$ then, by \eqref{kbxjhsvhchcvkjsz}, \eqref{hgjgfxhfdwfdsz} holds. 
If $X\subseteq C$ then, by \eqref{mlioniuib} applied for $X$, \eqref{hgjgfxhfdwfdsz} holds. We suppose from now on that  $X\cap C\neq\emptyset\neq X-C.$ Let {\boldmath$Z$} $=P_F^{X-C}$, {\boldmath$Y$} $=Z\cap P_F^C$ and  {\boldmath$W$} $=Y\cup(X\cap C)$. Then $X\cap Z=X-C, P_F^C\cap (X\cup Z)=W$ and $P_F^C\cup (X\cup Z)=P_F^X,$ $e_{E\cup A}(Y)=0.$ Thus, by \eqref{kbxjhsvhchcvkjsz} for $X-C$, \eqref{mlioniuib} for $W$ and the intersecting supermodularity of $h$, first for $X$ and $Z$, and then for $P_F^C$ and $X\cup Z,$ we have 
\begin{eqnarray*}
d^-_{\vec{E}\cup A}(X)	&	\ge	 &	d^-_{\vec{E}_1\cup A}(X-C)+d^-_{\vec{E}_2\cup A}(W)	\\
					&	\ge 	&	(h(X-C)-h(Z))+ (h(W)-h(P_F^C))					\\
					&	\ge	&	(h(X)-h(X\cup Z))+(h(X\cup Z)-h(P_F^X))				\\
					&	=	&	h(X)-h(P_F^X),
\end{eqnarray*}
so \eqref{hgjgfxhfdwfdsz} holds.
\end{proof}

We mention that Theorem \ref{hypgraphcoveringhorient} and hence Theorem \ref{lknvlknlkklsz} also work for mixed hypergraphs. This shows that the result of H\"orsch, Szigeti \cite{HSz2} can also be obtained from a theorem of Fortier et al. \cite{FKLSzT} on matroid-reachability-based packing of hyperarborescences.
\medskip

For the sake of completeness we show that Theorems \ref{hypgraphcoveringhorient} and \ref{lknvlknlkklsz} are in fact equivalent. We have just seen that Theorem \ref{hypgraphcoveringhorient} implies Theorem \ref{lknvlknlkklsz}. Now let us see the other direction. 

\begin{proof}
Let $(G=(V,E),h)$ be an instance of Theorem \ref{hypgraphcoveringhorient} such that $h(V)=0$ and \eqref{condhyph} holds. If $G$ is connected then Theorem \ref{lknvlknlkklsz} for $(G,h)$ reduces to  Theorem \ref{hypgraphcoveringhorient} because $h(P_G^X)=h(V)=0$ for every $\emptyset\neq X\subseteq V$. If the number {\boldmath$k$} of connected components of $G$ is larger than one then we need some more effort. Let {\boldmath$\ell$} $=\max\{h(X)+h(Y)-h(X\cup Y):X,Y\subset V, X\cap Y=\emptyset\}$ and {\boldmath$m$} $=\max\{k,\ell\}.$ Let {\boldmath$G'$} $=(V',E')$ be obtained from $G$ by adding a new vertex {\boldmath$s$} and by connecting $s$ to  $V$ by $m$ edges such that 
$G'$ is connected. Since $m\ge k$, this is possible. Let {\boldmath$h'$} be defined as follows: $h'(s)=m$ otherwise $h'(X)=h(X-s)$ for every $s\neq X\subseteq V'.$ Then $h'$ is an integer-valued intersecting supermodular set function on $V'.$ Indeed, the only case that must be checked is for pairs $X'=X\cup s$ and $Y'=Y\cup s$ with $\emptyset\neq X,Y\subseteq V$ and $X\cap Y=\emptyset.$ Then $h'(X')+h'(Y')-h'(X'\cup Y')=h(X)+h(Y)-h(X\cup Y)\le \ell\le m=h'(s)=h'(X'\cap Y')$. Note that since   $G'$ is connected and $h(V)=0$, we have
	\begin{eqnarray}
		h'(P^X_{G'})=h'(V')=h(V)=0 \text{ for every } \emptyset\neq X\subseteq V'.\label{hhvz}
	\end{eqnarray}
	
We now show that \eqref{condhyph} implies \eqref{mglvjhdfufehomf} for $(G',h')$. We only need  to check \eqref{mglvjhdfufehomf} for  subpartitions $\mathcal{P}$ of $V'$ because $G'$ is connected. If $\{s\}\in\mathcal{P}$ then, by \eqref{condhyph} and \eqref{hhvz}, we have  $e_{E'}(\mathcal{P})=m+e_{E}(\mathcal{P}-\{s\})\ge h'(s)+\sum_{X\in\mathcal{P}-\{s\}}h(X)=\sum_{X\in\mathcal{P}}h'(X)=\sum_{X\in\mathcal{P}}(h'(X)-h'(P^{V'}_{G'})).$ Otherwise, by \eqref{condhyph} and \eqref{hhvz}, we have $e_{E'}(\mathcal{P})\ge e_E(\mathcal{P})\ge\sum_{X\in\mathcal{P}}h(X-s)=\sum_{X\in\mathcal{P}}h'(X)=\sum_{X\in\mathcal{P}}(h'(X)-h'(P^{V'}_{G'}))$ so \eqref{mglvjhdfufehomf} holds for $(G',h').$

We can hence apply Theorem \ref{lknvlknlkklsz} for $(G',h')$ to get an orientation $\vec{G'}=(V',\vec{E'})$ such that 
	\begin{eqnarray}
		d^-_{\vec{E'}}(X)\ge h'(X)-h'(P^X_{G'}) \text{ for every } \emptyset\neq X\subseteq V'.\label{loigoigi}
	\end{eqnarray}

Deleting $s$ from $\vec{G'}$ we obtain an orientation $\vec{G}=(V,\vec E)$ of $G.$ We conclude by showing that   $\vec{G}$ satisfies \eqref{coverh}. 
By applying \eqref{loigoigi} for $s$ and \eqref{hhvz}, we get that $m=d_E(s)\ge d^-_{\vec{E'}}(s)\ge h'(s)-h'(P^{s}_{G'})=m$ so equality holds everywhere. In particular, $d_E(s)=d^-_{\vec{E'}}(s)$. Hence all the edges from $s$ to $V$ are oriented from $V$ to $s,$ thus $d^-_{\vec E}(X)=d^-_{\vec{E'}}(X)$  for every $X\subseteq V.$   Then, by \eqref{loigoigi} and \eqref{hhvz}, we have $d^-_{\vec E}(X)=d^-_{\vec{E'}}(X)\ge h'(X)-h'(P^X_{G'})=h(X)$ for every $\emptyset\neq X\subseteq V,$ so \eqref{coverh} holds.  This completes the proof.
\end{proof}

We conclude this section by showing that Theorem \ref{lblbkjlk}  easily follows from Theorems \ref{thmCsaba}  and \ref{lknvlknlkklsz}.

\begin{proof}
Let $(F,S,{\sf M})$ be an instance of Theorem \ref{lblbkjlk} that satisfies \eqref{lknemlkbfljevdhue}. Then, for $h(X)=-r_{{\sf M}}(S_X)$ for all $X\subseteq V$, \eqref{mglvjhdfufehomf} holds, so by Theorem \ref{lknvlknlkklsz} applied for $(F,h)$, there exists an orientation  $\vec E$ of $E$ such that in $\vec{F}=(V,\vec{E}\cup A)$ \eqref{hgjgfxhfdwfdsz} holds. Let $X\subseteq V.$ Since $P_{\vec{F}}^X\subseteq P_F^X$ and $r_{\sf M}$ is non-decreasing, we have $r_{{\sf M}}(S_{P_{\vec{F}}^X})\le r_{{\sf M}}(S_{P_F^X}).$ By \eqref{hgjgfxhfdwfdsz} applied for $P_{\vec{F}}^X$, we have $r_{{\sf M}}(S_{P_{\vec{F}}^X})\ge r_{{\sf M}}(S_{P_F^X})$. Hence  $r_{{\sf M}}(S_{P_{\vec{F}}^X})=r_{{\sf M}}(S_{P_F^X}).$  Thus \eqref{hgjgfxhfdwfdsz} implies that \eqref{csabicond} holds in $(\vec{F},S,{\sf M}).$ Then, by Theorem \ref{thmCsaba}, there exists a  matroid-reachability-based packing of arborescences in $(\vec{F},S,{\sf M}).$ Since $r_{{\sf M}}(S_{P_{\vec{F}}^X})=r_{{\sf M}}(S_{P_F^X})$, by replacing the arcs in $\vec E$ by the edges in $E$, we obtain a  matroid-reachability-based packing of mixed arborescences in $(F,S,{\sf M}).$
\end{proof}

\section{Packing  mixed hyperarborescences}\label{hyperarbor}

In this section we first list known results on  mixed hyperarborescences that are related to our second contribution on $(f,g)$-bounded $k$-regular $(\ell,\ell')$-limited packings of mixed hyperarborescences and we provide the proof of the new result.
\medskip

Theorem \ref{edmondsarborescences} was extended for the case when the roots of the arborescences are not fixed but the number of arborescences in the packing rooted at any vertex is bounded. 

 \begin{theo}[Frank \cite{FA}, Cai \cite{cai1}]\label{frankarborescencesroots} 
Let $D=(V,A)$ be a digraph, $f,g: V\rightarrow \mathbb Z_+$ functions and $k\in \mathbb{Z}_+.$
There exists an $(f,g)$-bounded packing of   $k$  spanning arborescences in $D$  if and only if 
\begin{eqnarray}
g(v)	& 	\geq 	&	f(v)\hskip 4.6truecm  \text{ for every } v \in V,\label{fg} \\ 
	e_{A}({\cal P})		&	\ge 	& 	k|\mathcal{P}|-\min\{k-f(\overline{\cup\mathcal{P}}),g(\cup\mathcal{P})\}\hskip .52truecm \text{ for every subpartition } \mathcal{P} \text{ of } V.\label{jvljh1}
\end{eqnarray}
\end{theo}

If $S$ is a multiset of  $V$ and $f(v)=g(v)=|S_v|$ for all $v\in V$ then  Theorem \ref{frankarborescencesroots} reduces to Theorem \ref{edmondsarborescences}.
\medskip

Theorem \ref{frankarborescencesroots} can be generalized for the case when the arborescences are not necessarily spanning but every vertex must belong to the same number of arborescences in the packing. For $g: V\rightarrow \mathbb Z_+$ and $k\in \mathbb Z_+,$ let {\boldmath$g_k$}$(v)=\min\{k,g(v)\}$ for every $v\in V.$ For convenience, we present not the original version of the result  of \cite{BF3} which is about packing spanning branchings but one that fits better to our framework.

\begin{theo}[B\'erczi, Frank \cite{BF3}]\label{bobisuv}
Let $D=(V,A)$ be a digraph, $f,g: V\rightarrow \mathbb Z_+$ functions and $k,\ell,\ell'\in\mathbb Z_+$. 
There exists an $(f,g)$-bounded $k$-regular $(\ell,\ell')$-limited packing of  arborescences in $D$  if and only if  
	\begin{eqnarray}  
		g_k(v)				&	\ge	&	f(v) \hskip 4.67truecm \text{ for every } v\in V,\label{bdkzjdju}\\
		\min\{g_k(V),\ell'\}		&	\ge	&	\ell\label{gkvl}\\
		e_A({\cal P})&\ge & k|\mathcal{P}|-\min\{\ell'-f(\overline{\cup\mathcal{P}}),g(\cup\mathcal{P})\} 
			\hskip .52truecm \text{ for every subpartition } \mathcal{P} \text{ of } V.\label{ndjvhd}
	\end{eqnarray}
\end{theo}

For $k=\ell=\ell'$, Theorem \ref{bobisuv} reduces to Theorem \ref{frankarborescencesroots}.
\medskip

Theorem \ref{edmondsarborescences} was generalized for dypergraphs as follows. We only need the special case when the multiset $S$ is equal to $k$ times a vertex $s$.

\begin{theo}[Frank, T. Kir\'aly, Z. Kir\'aly \cite{fkiki}]\label{hyperarborescences} 
Let $\mathcal{D}=(V,\mathcal{A})$ be a dypergraph, $s\in V$ and $k\in \mathbb{Z}_+.$
There exists a packing of   $k$  spanning $s$-hyperarborescences in $\mathcal{D}$ if and only if 
	\begin{eqnarray*} 
		 d^-_\mathcal{A}(X) &\geq &k \hskip .5truecm \text{ for every  $\emptyset\neq X\subseteq V-s.$}
	\end{eqnarray*}
\end{theo}

Theorem \ref{hyperarborescences} easily implies the following corollary that we will  apply in the proof of our new result.

\begin{cor} \label{thmSzhyp}
Let $\mathcal{D}=(V,\mathcal{A})$  be a dypergraph and $S$ a multiset of  $V$. There exists a  $k$-regular packing of $s$-hyperarborescences ($s\in S$) in $\mathcal{D}$  if and only if  
\begin{eqnarray}
|S_v|					  &	\le	& k \hskip 1.55truecm\text{ for every } v\in V, \label{matcondori1}\\
d^-_{\mathcal{A}}(X)&	\geq	& k-|S_X| \hskip .44truecm\text{ for every } \emptyset\neq X\subseteq V.\label{matcondori2hyp}
\end{eqnarray}
\end{cor}

A common extension of Theorems  \ref{mixedgraph}  and \ref{frankarborescencesroots} was provided by Gao, Yang \cite{GY2}.

 \begin{theo}[Gao, Yang \cite{GY2}]\label{gaoyang3} 
Let $F=(V,E\cup A)$ be a mixed graph, $f,g: V\rightarrow \mathbb Z_+$ functions, and $k\in \mathbb{Z}_+.$
There exists an $(f,g)$-bounded  packing of $k$  spanning mixed arborescences in $F$ if and only if \eqref{fg} holds and 
  \begin{eqnarray*}
	e_{E\cup A}({\cal P})\ge 	k|\mathcal{P}|-\min\{k-f(\overline{\cup\mathcal{P}}),g(\cup\mathcal{P})\} \text{ for every subpartition } \mathcal{P} \text{ of } V.
\end{eqnarray*}
\end{theo}

If $S$ is a multiset of  $V$ and $f(v)=g(v)=|S_v|$ for every $v\in V$ then  Theorem \ref{gaoyang3} reduces to Theorem \ref{mixedgraph}. If $F$ is a digraph then Theorem \ref{gaoyang3} reduces to Theorem \ref{frankarborescencesroots}. 
\medskip

Theorem \ref{gaoyang3} can be generalized for mixed hypergraphs.

\begin{theo}[H\"orsch, Szigeti \cite{HSz}]\label{hsz} 
Let $\mathcal{F}=(V,\mathcal{E}\cup \mathcal{A})$ be a mixed hypergraph, $f,g: V\rightarrow \mathbb Z_+$ functions, and $k\in \mathbb{Z}_+.$
There exists an $(f,g)$-bounded  packing of $k$  spanning mixed hyperarborescences in $\mathcal{F}$ if and only if \eqref{fg} holds and  
\begin{eqnarray*}
	e_{\mathcal{E}\cup \mathcal{A}}({\cal P})		&	\ge 	& 	k|\mathcal{P}|-\min\{k-f(\overline{\cup\mathcal{P}}),g(\cup\mathcal{P})\}\hskip .52truecm \text{ for every subpartition } \mathcal{P} \text{ of } V.
\end{eqnarray*}
\end{theo}

If $\mathcal{F}$ is a mixed graph then Theorem \ref{hsz} reduces to Theorem \ref{gaoyang3}. 
\medskip

The main contribution of the present paper is a common generalization of Theorems \ref{bobisuv} and \ref{hsz}.

\begin{theo}\label{sibevhz2} 
Let $\mathcal{F}=(V,\mathcal{E}\cup \mathcal{A})$ be a mixed hypergraph, $f,g: V\rightarrow \mathbb Z_+$ functions, and $k,\ell,\ell'\in \mathbb{Z}_+-\{0\}.$
There exists an $(f,g)$-bounded $k$-regular $(\ell,\ell')$-limited packing of 
 mixed hyperarborescences in $\mathcal{F}$ if and only if \eqref{bdkzjdju} and \eqref{gkvl}  hold and  
\begin{eqnarray}
	e_{\mathcal{E}\cup \mathcal{A}}({\cal P})		&	\ge 	& 	k|\mathcal{P}|-\min\{\ell'-f(\overline{\cup\mathcal{P}}),g_k(\cup\mathcal{P})\}\hskip .52truecm \text{ for every subpartition } \mathcal{P} \text{ of } V.\label{jvljh1hyp2}
\end{eqnarray}
\end{theo}

If $\mathcal{F}$ is a  digraph then Theorem \ref{sibevhz2} reduces to Theorem \ref{bobisuv}.
If $k=\ell=\ell'$ then Theorem \ref{sibevhz2} reduces to Theorem \ref{hsz}.
Theorem \ref{sibevhz2} will be obtained from the theory of generalized polymatroids and some matroid construction for mixed hypergraphs. We now explain these concepts.
\medskip

 Generalized polymatroids were introduced by Hassin \cite{hassin} and independently by  Frank \cite{fgp}. 
  For a pair $(p,b)$ of set functions on $S$ and $\alpha,\beta\in \mathbb{R},$ let us introduce the polyhedra 
 \begin{eqnarray*}
 \text{\boldmath$Q(p,b)$} &=&\{x\in \mathbb{R}^S: p(Z)\le x(Z)\le b(Z)\ \text{ for all } Z\subseteq S\},\\
  \text{\boldmath$K(\alpha,\beta)$}& =&\{x\in \mathbb{R}^S: \alpha\le x(S)\le \beta\}.
\end{eqnarray*}
 If  $p(\emptyset)=b(\emptyset)=0,$ $p$ is supermodular, $b$ is submodular and 
 $b(X)-p(Y)\ge b(X-Y)-p(Y-X)$ for all $X,Y\subseteq S,$ 
 the polyhedron $Q(p,b)$ is called a {\it generalized-polymatroid}, shortly {\it g-polymatroid}. 
The polyhedron $K(\alpha,\beta)$ is called a {\it plank.} 
 The Minkowski sum  of the $n$ g-polymatroids $Q(p_i,b_i)$ is denoted by {\boldmath$\sum_1^nQ(p_i, b_i)$}.
 We will need the following results on g-polymatroids, for more details  see \cite{book}.

\begin{theo}[Frank \cite{book}]\label{gpmip}
The following hold:	
	\begin{enumerate}
		\item Let $Q(p,b)$ be a g-polymatroid, $K(\alpha,\beta)$ a plank and $M$ $=Q(p,b)\cap K(\alpha,\beta).$\label{e}
			\begin{itemize}
				\item[(i)] $M\neq\emptyset$ if and only if $p\le b$, $\alpha\le \beta,$ $p(S)\le\beta$ and $\alpha\le b(S).$ 
				\item[(ii)] $M$ is a $g$-polymatroid. 
				\item[(iii)] If $M\neq\emptyset$ then $M=Q(p^\alpha_\beta, b^\alpha_\beta)$ with
					\begin{eqnarray}\label{interplank}
						\text{{\boldmath $p^\alpha_\beta(Z)$}} =\max\{p(Z), \alpha-b(S-Z)\}, 
						\ \ \text{{\boldmath $b^\alpha_\beta(Z)$}} =\min\{b(Z), \beta-p(S-Z)\}.
					\end{eqnarray}
			\end{itemize}
		\item Let $Q(p_1,b_1)$ and $Q(p_2,b_2)$ be two non-empty g-polymatroids and $M=Q(p_1,b_1)\cap Q(p_2,b_2).$\label{g}
			\begin{itemize}
				\item[(i)] $M\neq\emptyset$  if and only if $p_1\le b_2$ and $p_2\le b_1$. 
				\item[(ii)] If $p_1,b_1,p_2,b_2$ are integral and $M\neq\emptyset$ then $M$ contains an integral element.
			\end{itemize}
		\item Let $Q(p_i,b_i)$ be $n$ non-empty g-polymatroids. Then $\sum_1^nQ(p_i, b_i)=Q(\sum_1^np_i, \sum_1^nb_i).$\label{i}
	\end{enumerate}
\end{theo}
 
Given a hypergraph $\mathcal{H}=(V,\mathcal{E})$, let 
{\boldmath$\mathcal{I}_{\mathcal{H}}$} $=\{\mathcal{Z}\subseteq \mathcal{E}:|V(\mathcal{Z}')|>|\mathcal{Z}'| \text{ for all } \emptyset\neq \mathcal{Z}'\subseteq \mathcal{Z}\}.$ Lorea \cite{l} showed that  $\mathcal{I}_{\mathcal{H}}$  is the set of independent sets of a matroid {\boldmath${\sf M}_{\mathcal{H}}$} on $\mathcal{E}$, called the {\it hypergraphic matroid} of the hypergraph $\mathcal{H}$.
We also need the {\it $k$-hypergraphic matroid} {\boldmath${\sf M}_{\mathcal{H}}^k$} of ${\mathcal{H}}$ which is the $k$-sum matroid of ${\sf M}_{\mathcal{H}}$, that is the matroid on ground set $\mathcal{E}$ in which a subset of $\mathcal{E}$ is independent if it can be partitioned into $k$  independent sets of ${\sf M}_{\mathcal{H}}$. 
H\"orsch, Szigeti \cite{HSz} extended the previous construction for mixed hypergraphs as follows. Let $\mathcal{F}=(V,\mathcal{A}\cup \mathcal{E})$ be a mixed hypergraph. For a subpartition $\mathcal{P}$ of $V,$ {\boldmath$\mathcal{A}(\mathcal{P})$} and {\boldmath$\mathcal{E}(\mathcal{P})$} denote the set of dyperedges and the set of hyperedges that enter some member of $\mathcal{P}$. Let {\boldmath$\mathcal{H}_{\mathcal{F}}$} $=(V,$ {\boldmath$\mathcal{E}_{\mathcal{A}}$}$\cup \mathcal{E})$  the underlying hypergraph of $\mathcal{F}$ and  {\boldmath$\mathcal{D}_{\mathcal{F}}$} $=(V,\mathcal{A}\cup \mathcal{A}_{\mathcal{E}})$ the {\it directed extension}  of $\mathcal{F}$ where {\boldmath$\mathcal{A}_{\mathcal{E}}$} $=\bigcup_{e\in\mathcal{E}}\mathcal{A}_e$ and for $e\in\mathcal{E},$ {\boldmath$\mathcal{A}_e$}$=\{(e-x,x):x \in e\}$.
The {\it extended $k$-hypergraphic matroid} {\boldmath${\sf M}_{\mathcal{F}}^k$}  of ${\mathcal{F}}$ on $\mathcal{A}\cup\mathcal{A}_{\mathcal{E}}$ is obtained from ${\sf M}_{\mathcal{H}_{\mathcal{F}}}^k$ by replacing every $e \in \mathcal{E}$ by $|e|$ parallel copies of itself, associating these elements to the dyperedges in $\mathcal{A}_e$ and associating every hyperedge of $\mathcal{E}_\mathcal{A}$ to the corresponding dyperedge in $\mathcal{A}$. 
It is shown in \cite{HSz} that the rank function of the extended $k$-hypergraphic matroid ${\sf M}_{\mathcal{F}}^k$ satisfies for all  $\mathcal{Z}\subseteq\mathcal{A}\cup\mathcal{A}_{\mathcal{E}},$

\begin{equation}\label{kldnkjebdjheu1}
r_{{\sf M}_{\mathcal{F}}^k}(\mathcal{Z})=\min\{|\mathcal{Z}\cap \mathcal{A}(\mathcal{P})|+|\{e\in\mathcal{E}(\mathcal{P}):{\mathcal{Z}}\cap \mathcal{A}_e\neq\emptyset\}|+k(|V|-|\mathcal{P}|): \mathcal{P} \text{ partition of } V\}.
\end{equation}

Theorem \ref{sibevhz2} will follow from the following lemma.

\begin{lemma}\label{jbdevdouycdit}
Let $\mathcal{F}=(V,\mathcal{E}\cup \mathcal{A})$ be a mixed hypergraph, $f,g: V\rightarrow \mathbb Z_+$ functions, and $k,\ell,\ell'\in \mathbb{Z}_+-\{0\}.$ Let {\boldmath${\sf M}_v$} $=(\rho_{\mathcal{A}\cup\mathcal{A}_{\mathcal{E}}}(v),r_v)$ be the free matroid  for all $v\in V$ and {\boldmath${\sf M}_{\mathcal{F}}^k$} the extended $k$-hypergraphic matroid  of ${\mathcal{F}}$ on $\mathcal{A}\cup\mathcal{A}_{\mathcal{E}}.$ Let us define the following polyhedron $$\text{\boldmath$T$} =(\sum_{v\in V}(Q(0,r_v)\cap K(k-g_k(v),k-f(v))))\cap K(k|V|-\ell',k|V|-\ell)\cap Q(0,r_{{\sf M}_{\mathcal{F}}^k}).$$
\begin{itemize}
	\item[(a)]   The characteristic vectors of the dyperedge sets of  the $(f,g)$-bounded $k$-regular $(\ell,\ell')$-limited packings of  
	hyperarborescences in orientations of $\mathcal{F}$ are exactly the integer points of  $T.$
	\item[(b)] 	$T\neq\emptyset$ if and only if  \eqref{bdkzjdju} and \eqref{gkvl}  hold and for every $\mathcal{Z}\subseteq\mathcal{A}\cup\mathcal{A}_{\mathcal{E}},$
	\begin{eqnarray}
		\sum_{v\in V}\max\{0,k-g_k(v)-d_\mathcal{Z}^-(v)\}	&	\le	&	r_{{\sf M}_{\mathcal{F}}^k}(\overline{\mathcal{Z}}),\label{jbipuib1}\\
				k|V|-\ell'-\sum_{v\in V}\min\{d_\mathcal{Z}^-(v),k-f(v)\}	&	\le	&	r_{{\sf M}_{\mathcal{F}}^k}(\overline{\mathcal{Z}}).\label{jbipuib2}
	\end{eqnarray}
	\item[(c)] \eqref{jbipuib1} and \eqref{jbipuib2} are equivalent to \eqref{jvljh1hyp2}.
\end{itemize}
\end{lemma}

\begin{proof}
{\bf (a)} To prove the {\bf necessity}, let {\boldmath$\mathcal{B}$} be an $(f,g)$-bounded $k$-regular $(\ell,\ell')$-limited packing of  hyperarborescences in an orientation {\boldmath$\vec{\mathcal{F}}$} of $\mathcal{F}.$  
Let {\boldmath$S$} be the root set of the hyperarborescences in $\mathcal{B}$ and {\boldmath$\vec{\mathcal{Z}}$}  the dyperedge set of $\mathcal{B}$. Since the packing is $(f,g)$-bounded $k$-regular $(\ell,\ell')$-limited, we have 
	\begin{eqnarray}
		f(v)	&	\le	&	|S_v|\hskip .24truecm \le \hskip .24truecm g_k(v) \hskip .24truecm \text{ for all } v\in V,\label{boun}\\
		k-d_{\vec{\mathcal{Z}}}^-(v)	&	=	&	|S_v| \hskip 2truecm \text{ for all } v\in V,\label{reg}\\
		\ell	&	\le	&	\hskip .06truecm|S|\hskip .33truecm \le\hskip .24truecm \ell'.\label{lim}
	\end{eqnarray}
By \eqref{reg}, we get 
	\begin{eqnarray}
		k|V|-|\vec{\mathcal{Z}}|=\sum_{v\in V}(k-d_{\vec{\mathcal{Z}}}^-(v))=\sum_{v\in V}|S_v|=|S|.\label{sum}
	\end{eqnarray}
Let  {\boldmath$m$} be  the characteristic vector of  $\vec{\mathcal{Z}}$  and {\boldmath$m_v$}  the restriction of $m$ on $\rho_{\mathcal{A}\cup\mathcal{A}_{\mathcal{E}}}(v)$  for all $v\in V$. 
Then $m_v$ is a characteristic vector, so 
	\begin{eqnarray}
		m_v\in Q(0,r_v)  \text{ for all }  v\in V. \label{mvrv}
	\end{eqnarray}
By \eqref{boun}, \eqref{reg} and $d_{\vec{\mathcal{Z}}}^-(v)=m_v(\rho_{\mathcal{A}\cup\mathcal{A}_{\mathcal{E}}}(v))$ for all $v\in V$, we obtain that 
	\begin{eqnarray}
		m_v\in K(k-g_k(v),k-f(v)) \text{ for all }  v\in V. \label{kgkfjj}
	\end{eqnarray}
It follows, by \eqref{mvrv} and \eqref{kgkfjj},  that 
	\begin{eqnarray}
		m\in \sum_{v\in V}(Q(0,r_v)\cap K(k-g_k(v),k-f(v))).\label{mqk}
	\end{eqnarray}
By \eqref{lim}, \eqref{sum}, and $|\vec{\mathcal{Z}}|=m(\mathcal{A}\cup\mathcal{A}_{\mathcal{E}}),$ we obtain that  
	\begin{eqnarray}
		m\in K(k|V|-\ell',k|V|-\ell).\label{kkvlkv}
	\end{eqnarray}
	
\begin{claim}	\label{bdnkls}
$\vec{\mathcal{Z}}$ is independent in  ${\sf M}_{{\mathcal{F}}}^k$.
\end{claim}

\begin{proof}
We first show that $\vec{\mathcal{Z}}$ is the dyperedge set of a  packing of $k$ spanning hyperbranchings  in $\vec{\mathcal{F}}$. Indeed, let {\boldmath$\vec{\mathcal{G}}$} be the dypergraph with vertex set $V\cup\{s\}$ where {\boldmath$s$} is a new vertex, and dyperedge set $\vec{\mathcal{Z}}\cup A$ where {\boldmath$A$} $=\{ss':s'\in S\}.$
Since $\vec{\mathcal{Z}}$ is the dyperedge set of a  $k$-regular packing of  hyperarborescences in $\vec{\mathcal{F}}$, by Corollary \ref{thmSzhyp}, we have $d^-_{\vec{\mathcal{Z}}\cup A}(X)=d^-_{\vec{\mathcal{Z}}}(X)+d^-_A(X)=d^-_{\vec{\mathcal{Z}}}(X)+|S_X|\ge k.$
Then, by Theorem \ref{hyperarborescences}, there exists a packing of $k$ spanning $s$-hyperarborescences in $\vec{\mathcal{G}}$. By deleting the vertex $s$ from each $s$-hyperarborescence in the packing we obtain a  packing of $k$ spanning hyperbranchings  in $\vec{\mathcal{F}}$ with dyperedge set $\vec{\mathcal{Z}}'$. Since $\vec{\mathcal{Z}}\supseteq\vec{\mathcal{Z}}'$ and $|\vec{\mathcal{Z}}'|\ge k|V|-|S|=|\vec{\mathcal{Z}}|,$ we get that $\vec{\mathcal{Z}}=\vec{\mathcal{Z}}'.$ Hence $\vec{\mathcal{Z}}$ is the dyperedge set of a  packing of $k$ spanning hyperbranchings  in $\vec{\mathcal{F}}$.

For any hyperbranching, the number of its vertices is at least the number of heads of its dyperedges plus the number of its roots  and hence strictly larger than the number of its dyperedges.
Thus the hyperedge set of the underlying hypergraph of each hyperbranching is independent in ${\sf M}_{\mathcal{H}_{\mathcal{F}}}$. As  the hyperbranchings in the packing are dyperedge disjoint, it follows that $\mathcal{Z}$ is independent in ${\sf M}_{\mathcal{H}_{\mathcal{F}}}^k$. Then, since $\vec{\mathcal{Z}}$ is in the orientation $\vec{\mathcal{F}}$ of $\mathcal{F}$, $\vec{\mathcal{Z}}$ is independent in  ${\sf M}_{{\mathcal{F}}}^k$.  
\end{proof}

By Claim \ref{bdnkls} and since $m$ is  the characteristic vector of $\vec{\mathcal{Z}}$, we get that 
	\begin{eqnarray}
		m\in Q(0,r_{{\sf M}_{\mathcal{F}}^k}).\label{qrmfk}
	\end{eqnarray}
It follows, by \eqref{mqk}, \eqref{kkvlkv}, and \eqref{qrmfk}, that $m$ is an integer point of  $T.$
\medskip
   
To prove the {\bf sufficiency}, let {\boldmath$m$} $=(m_v)_{v\in V}$ be an  integer point of $T,$ that is $m_v\in Q(0,r_v)\cap K(k-g_k(v),k-f(v))$ for all $v\in V$ and $m\in K(k|V|-\ell',k|V|-\ell)\cap Q(0,r_{{\sf M}_{\mathcal{F}}^k}).$ 
Since $m_v$ is an integer point in $Q(0,r_v),$ $m_v$ is the characteristic vector of a subset {\boldmath $\vec{\mathcal{Z}}_v$} of $\rho_{\mathcal{A}\cup\mathcal{A}_{\mathcal{E}}}(v).$ 
Since $m_v\in K(k-g_k(v),k-f(v)),$ we have 
\begin{eqnarray}\label{dkjshjd}
k-g_k(v)\le m_v(\rho_{\mathcal{A}\cup\mathcal{A}_{\mathcal{E}}}(v))=|\vec{\mathcal{Z}}_v|=m_v(\rho_{\mathcal{A}\cup\mathcal{A}_{\mathcal{E}}}(v))\le k-f(v).
 \end{eqnarray}
Let {\boldmath$\vec{\mathcal{Z}}$} $=\bigcup_{v\in V}\vec{\mathcal{Z}}_v.$ Note that $d_{\vec{\mathcal{Z}}}^-(v)=|\vec{\mathcal{Z}}_v|$ for all $v\in V.$ Then, by $f\ge 0,$ we have $k-d_{\vec{\mathcal{Z}}}^-(v)\ge f(v)\ge 0$ for all $v\in V.$
Since $m\in K(k|V|-\ell',k|V|-\ell)$, we have 
\begin{eqnarray}\label{enfiuebfuye}
k|V|-\ell'\le m({\mathcal{A}\cup\mathcal{A}_{\mathcal{E}}})=|\vec{\mathcal{Z}}|=m({\mathcal{A}\cup\mathcal{A}_{\mathcal{E}}})\le k|V|-\ell.
\end{eqnarray}
Since $m\in  Q(0,r_{{\sf M}_{\mathcal{F}}^k}),$ we get that $\vec{\mathcal{Z}}$ is independent in ${\sf M}_{\mathcal{F}}^k.$ 
It follows that $\vec{\mathcal{Z}}$ is a subset of the dyperedge set of an orientation {\boldmath$\vec{\mathcal{F}}$} of $\mathcal{F}$ and 
in the hypergraph {\boldmath$\mathcal{H}_\mathcal{F}$} $=(V,\mathcal{E}_\mathcal{A}\cup\mathcal{E})$ we have for all $X\subseteq V,$ 
\begin{eqnarray}\label{lnbzudvuzy}
|\mathcal{Z}(X)|\le r_{{\sf M}_{\mathcal{H}_\mathcal{F}}^k}(\mathcal{Z}(X))\le k(|X|-1).
\end{eqnarray}
Let {\boldmath$S$} be the multiset of $V$ such that
$|S_v|=k-d_{\vec{\mathcal{Z}}}^-(v)$ for all $v\in V.$  Since $k-d_{\vec{\mathcal{Z}}}^-(v)\ge 0$ for all $v\in V,$ $S$ exists.
As $d^-_{\vec{\mathcal{Z}}}\ge 0,$ \eqref{matcondori1} holds.
Since for all $X\subseteq V,$ by \eqref{lnbzudvuzy}, we have $$d_{\vec{\mathcal{Z}}}^-(X)=\sum_{v\in X}d_{\vec{\mathcal{Z}}}^-(v)-|\vec{\mathcal{Z}}(X)|=\sum_{v\in X}(k-|S_v|)-|\mathcal{Z}(X)|\ge k|X|-|S_X|-k(|X|-1)=k-|S_X|,$$ so \eqref{matcondori2hyp} holds for {\boldmath$\vec{\mathcal{F}}'$} $=(V,\vec{\mathcal{Z}})$. Then, by Corollary \ref{thmSzhyp}, there exists a $k$-regular packing of $s$-hyperarborescen\-ces $(s\in S)$  in $\vec{\mathcal{F}}'$ and hence  in $\vec{\mathcal{F}}$. 
Since the number of dyperedges in the packing is $k|V|-|S|=\sum_{v\in V}(k-|S_v|)=\sum_{v\in V}d_{\vec{\mathcal{Z}}}^-(v)=|\vec{\mathcal{Z}}|$, the dyperedge set of the packing is $\vec{\mathcal{Z}}.$
As  for all $v\in V,$ by \eqref{dkjshjd}, we have $$f(v)\le k-|{\vec{\mathcal{Z}}_v}|=k-d_{\vec{\mathcal{Z}}}^-(v)=|S_v|=k-d_{\vec{\mathcal{Z}}}^-(v)=k-|{\vec{\mathcal{Z}}_v}|\le g_k(v)\le g(v),$$ so the packing is $(f,g)$-bounded.
Since, by \eqref{enfiuebfuye}, we have $$\ell\le k|V|-|\vec{\mathcal{Z}}|=|S|=k|V|-|\vec{\mathcal{Z}}|\le \ell',$$ so the packing is $(\ell,\ell')$-limited.
Finally, as $\vec{\mathcal{F}}$ is an orientation  of $\mathcal{F}$, the proof is complete.
\medskip

{\bf (b)} By Theorem \ref{gpmip}.\ref{e}, for all $v\in V$, $Q(0,r_v)\cap K(k-g_k(v),k-f(v))\neq\emptyset$ if and only if $0\le r_v$ (that always holds), $k-g_k(v)\le k-f(v)$ (that is \eqref{bdkzjdju} holds), $0\le k-f(v)$ (that holds by the previous inequality) and $k-g_k(v)\le r_v(\rho_{\mathcal{\mathcal{A}\cup\mathcal{A}_{\mathcal{E}}}}(v))=d_{\mathcal{\mathcal{A}\cup\mathcal{A}_{\mathcal{E}}}}^-(v).$ Then $Q(0,r_v)\cap K(k-g_k(v),k-f(v))=Q(p_v, b_v)$ where, by  \eqref{interplank}, we have for all $\mathcal{Z}_v\subseteq \rho_{\mathcal{\mathcal{A}\cup\mathcal{A}_{\mathcal{E}}}}(v),$

\begin{equation}\label{pabidyp}
\text{{\boldmath $p_v$}}(\mathcal{Z}_v) =\max\{0,k-g_k(v)-d_{\overline{\mathcal{Z}}_v}^-(v)\}, \ \ 
\text{{\boldmath $b_v$}}(\mathcal{Z}_v) =\min\{d_{\mathcal{Z}_v}^-(v), k-f(v)\}.
\end{equation}
By Theorem \ref{gpmip}.\ref{i}, $\sum_{v\in V}Q(p_v, b_v)=Q(p_\Sigma, b_\Sigma)$ where 
\begin{equation}\label{idgziufiu}
\text{{\boldmath $p_\Sigma$}} =\sum_{v\in V}p_v \text{ and {\boldmath $b_\Sigma$}} =\sum_{v\in V}b_v.
\end{equation}
By Theorem \ref{gpmip}.\ref{e}, $Q(p_\Sigma, b_\Sigma)\cap K(k|V|-\ell',k|V|-\ell)\neq\emptyset$ if and only if $Q(p_v, b_v)\neq\emptyset$ for all $v\in V$, $k|V|-\ell'\le k|V|-\ell$ (which is equivalent to one of the conditions in \eqref{gkvl}), $p_\Sigma(\mathcal{A}\cup\mathcal{A}_{\mathcal{E}})\le k|V|-\ell$  (which, by $p_\Sigma(\mathcal{A}\cup\mathcal{A}_{\mathcal{E}})=\sum_{v\in V}p_v(\rho_{\mathcal{\mathcal{A}\cup\mathcal{A}_{\mathcal{E}}}}(v))=\sum_{v\in V}\max\{0,k-g_k(v)-d_{\overline{\mathcal{\mathcal{A}\cup\mathcal{A}_{\mathcal{E}}}}}^-(v)\}=\sum_{v\in V}(k-g_k(v))=k|V|-g_k(V)$, is equivalent to the other condition in \eqref{gkvl}) and $b_\Sigma(\mathcal{A}\cup\mathcal{A}_{\mathcal{E}})\ge k|V|-\ell'$. Then the intersection is equal to a generalized polyhedron $Q(p, b)$ where, by \eqref{interplank}, \eqref{pabidyp},  and \eqref{idgziufiu},  for all $\mathcal{Z}\subseteq {\mathcal{A}\cup\mathcal{A}_{\mathcal{E}}}$, 
we have
\begin{eqnarray}
\text{{\boldmath $p$}} (\mathcal{Z}) &=&
\max\left\{\sum_{v\in V}\max\{0,k-g_k(v)-d_{\overline{\mathcal{Z}}}^-(v)\},k|V|-\ell'-\sum_{v\in V}\min\{d_{\overline{\mathcal{Z}}}^-(v), k-f(v)\}\right\},\label{jbivuyodcyc} \\
\text{{\boldmath $b$}}(\mathcal{Z})	&=&
\min\left\{\sum_{v\in V}\min\{d_{\mathcal{Z}}^-(v), k-f(v)\}, k|V|-\ell-\sum_{v\in V} \max\{0,k-g_k(v)-d_{\mathcal{Z}}^-(v)\}\right\}.
\end{eqnarray}
By Theorem \ref{gpmip}.\ref{g}, $T=Q(p, b)\cap Q(0,r_{{\sf M}_{\mathcal{F}}^k})\neq\emptyset$ if and only if $Q(p, b)\neq\emptyset$, 
$p\le r_{{\sf M}_{\mathcal{F}}^k}$ (which, by \eqref{jbivuyodcyc}, is equivalent to \eqref{jbipuib1} and \eqref{jbipuib2}), and $b\ge 0$ (which holds by $b\ge p\ge 0$). Note that $k-g_k(v)\le d_{\mathcal{\mathcal{A}\cup\mathcal{A}_{\mathcal{E}}}}^-(v)$ for all $v\in V$ and $b_\Sigma(\mathcal{A}\cup\mathcal{A}_{\mathcal{E}})\ge k|V|-\ell'$ follow from $p\le r_{{\sf M}_{\mathcal{F}}^k}$ applied for $\mathcal{Z}=\emptyset$ and the proof is complete.
\medskip

{\bf (c)} We note that \eqref{jbipuib1} is equivalent to 
	\begin{eqnarray}
				k|V|-g_k(V)-\sum_{v\in V}\min\{d_\mathcal{Z}^-(v),k-g_k(v)\}	&	\le	&	r_{{\sf M}_{\mathcal{F}}^k}(\overline{\mathcal{Z}}).\label{jbipuib3}
	\end{eqnarray}

First we show that \eqref{jbipuib1} and \eqref{jbipuib2} imply \eqref{jvljh1hyp2}. Let {\boldmath$\mathcal{P}$} be a subpartition of $V.$
Let {\boldmath$\mathcal{Z}$} $=\bigcup_{v\in \overline{\cup\mathcal{P}}}\rho_\mathcal{A}(v)\cup\bigcup_{e\in\mathcal{E}(\mathcal{F}(\overline{\cup\mathcal{P}}))}\mathcal{A}_e$ and {\boldmath$\mathcal{P}'$} $=\mathcal{P}\cup\{v\}_{v\in\overline{\cup\mathcal{P}}}.$ Note that $d_\mathcal{Z}^-(v)=0$ for all $v\in \cup\mathcal{P}$, 
\begin{equation}
	\sum_{v\in V}\min\{d_\mathcal{Z}^-(v),k-h(v)\}\le k|\overline{\cup\mathcal{P}}|-h(\overline{\cup\mathcal{P}}) \text{ for } h\in\{g_k,f\},\label{jbiebuiecipu}
\end{equation}
$\mathcal{P}'$ is a partition of $V$, and, by  \eqref{kldnkjebdjheu1}, we have
	\begin{equation}\label{nfokeoio}
r_{{\sf M}_{\mathcal{F}}^k}(\overline{\mathcal{Z}})\le |\overline{\mathcal{Z}}\cap \mathcal{A}(\mathcal{P}')|+|\{e\in\mathcal{E}(\mathcal{P}'):\overline{\mathcal{Z}}\cap \mathcal{A}_e\neq\emptyset\}|+k(|V|-|\mathcal{P}'|)=e_{\mathcal{A}\cup\mathcal{A}_{\mathcal{E}}}(\mathcal{P})+k(|V|-|\mathcal{P}|-|\overline{\cup\mathcal{P}}|).
	\end{equation}
Then \eqref{jbipuib3}, \eqref{jbiebuiecipu} applied for $h=g_k$ and \eqref{nfokeoio} imply $e_{\mathcal{E}\cup \mathcal{A}}({\cal P})\ge k|\mathcal{P}|-g_k(\cup\mathcal{P}).$
Similarly,	\eqref{jbipuib2}, \eqref{jbiebuiecipu} applied for $h=f$ and \eqref{nfokeoio} imply $e_{\mathcal{E}\cup \mathcal{A}}({\cal P})\ge k|\mathcal{P}|-\ell'+f(\overline{\cup\mathcal{P}}).$ Hence \eqref{jvljh1hyp2} follows.
\medskip

We now show that \eqref{jvljh1hyp2}  implies \eqref{jbipuib2} and \eqref{jbipuib3} (and hence \eqref{jbipuib1}). Let {\boldmath$\mathcal{Z}$} $\subseteq\mathcal{A}\cup\mathcal{A}_{\mathcal{E}}.$ By  \eqref{kldnkjebdjheu1}, there exists a partition {\boldmath$\mathcal{P}$} of $V$ such that for {\boldmath$\mathcal{K}$} $=\{e\in\mathcal{E}(\mathcal{P}):\overline{\mathcal{Z}}\cap \mathcal{A}_e\neq\emptyset\}$, we have

\begin{equation}\label{kldnkjebdjheu}
r_{{\sf M}_{\mathcal{F}}^k}(\overline{\mathcal{Z}})=|\overline{\mathcal{Z}}\cap \mathcal{A}(\mathcal{P})|+|\mathcal{K}|+k(|V|-|\mathcal{P}|).
\end{equation}
For  $h\in\{g_k,f\},$ let {\boldmath$\mathcal{P}_h$} $=\{X\in\mathcal{P}: d_\mathcal{Z}^-(v)\le k-h(v) \text{ for all } v\in X\}.$ Note that $\mathcal{P}_h$ is a subpartition of $V$ and for every $X\in\mathcal{P}-\mathcal{P}_h,$ there exists a vertex $v_X\in X$ such that $d_\mathcal{Z}^-(v_X)>k-h(v_X).$  By the definition of $\mathcal{K},$ we have 
\begin{equation}\label{lkdnzovducyxi}
\mathcal{A}_{\mathcal{E}(\mathcal{P}_h)-\mathcal{K}}\subseteq\mathcal{Z}\cap \mathcal{A}_{\mathcal{E}(\mathcal{P}_h)}.
\end{equation}

\begin{claim}\label{inequality}
$r_{{\sf M}_{\mathcal{F}}^k}(\overline{\mathcal{Z}})+\sum_{v\in V}\min\{d_\mathcal{Z}^-(v),k-h(v)\}\ge e_{\mathcal{E}\cup\mathcal{A}}(\mathcal{P}_h)-k|\mathcal{P}_h|-h(\overline{\cup\mathcal{P}_h})+k|V|.$
\end{claim}

\begin{proof}
By \eqref{kldnkjebdjheu}, the definitions of $\mathcal{P}_h$ and $v_X,$ $d_\mathcal{Z}^-\ge 0, k-h\ge 0$, \eqref{lkdnzovducyxi}, and $h\ge 0,$ we have
\begin{eqnarray*}
&&		r_{{\sf M}_{\mathcal{F}}^k}(\overline{\mathcal{Z}})+\sum_{v\in V}\min\{d_\mathcal{Z}^-(v),k-h(v)\}\\
&=&		|\overline{\mathcal{Z}}\cap \mathcal{A}(\mathcal{P})|+|\mathcal{K}|+k(|V|-|\mathcal{P}|)+\sum_{v\in\cup\mathcal{P}_h}		\min\{d_\mathcal{Z}^-(v),k-h(v)\}+\sum_{v\in\overline{\cup\mathcal{P}_h}}\min\{d_\mathcal{Z}^-(v),k-h(v)\}\\
&\ge&	|\overline{\mathcal{Z}}\cap \mathcal{A}(\mathcal{P}_h)|+\sum_{v\in\cup\mathcal{P}_h}d_\mathcal{Z}^-(v)+					\sum_{X\in\mathcal{P}-\mathcal{P}_h}\sum_{v\in X}\min\{d_\mathcal{Z}^-(v),k-h(v)\}+|\mathcal{K}|+k(|V|-|\mathcal{P}|)\\
&\ge&	|\overline{\mathcal{Z}}\cap \mathcal{A}(\mathcal{P}_h)|+|\mathcal{Z}\cap \mathcal{A}(\mathcal{P}_h)|+|\mathcal{Z}\cap \mathcal{A}_{\mathcal{E}(\mathcal{P}_h)}|+\sum_{X\in\mathcal{P}-\mathcal{P}_h}(k-h(v_X))+|\mathcal{K}|+k(|V|-|\mathcal{P}|)\\
&\ge&	|\mathcal{A}(\mathcal{P}_h)|+|\mathcal{A}_{\mathcal{E}(\mathcal{P}_h)-\mathcal{K}}|+\sum_{X\in\mathcal{P}-				\mathcal{P}_h}(k-h(X))+|\mathcal{K}|+k(|V|-|\mathcal{P}|)\\
&\ge&	e_{\mathcal{E}\cup\mathcal{A}}(\mathcal{P}_h)-|\mathcal{K}|+k(|\mathcal{P}|-|\mathcal{P}_h|)-							h(\overline{\cup\mathcal{P}_h})+|\mathcal{K}|+k(|V|-|\mathcal{P}|)\\
&\ge&	e_{\mathcal{E}\cup\mathcal{A}}(\mathcal{P}_h)-k|\mathcal{P}_h|-h(\overline{\cup\mathcal{P}_h})+k|V|,
\end{eqnarray*}
and the claim follows.
\end{proof}

Claim \ref{inequality}, applied for $h=f,$ and \eqref{jvljh1hyp2} provide that $r_{{\sf M}_{\mathcal{F}}^k}(\overline{\mathcal{Z}})+\sum_{v\in V}\min\{d_\mathcal{Z}^-(v),k-f(v)\}\ge k|V|-\ell'$, so \eqref{jbipuib2} holds. Similarly, Claim \ref{inequality}, applied for $h=g_k,$ and \eqref{jvljh1hyp2} provide that $r_{{\sf M}_{\mathcal{F}}^k}(\overline{\mathcal{Z}})+\sum_{v\in V}\min\{d_\mathcal{Z}^-(v),k-g_k(v)\}\ge k|V|-g_k(V)$, so \eqref{jbipuib3} holds. The proof of the theorem is complete.
\end{proof}

We finish the paper by showing that Theorem \ref{gpmip} and Lemma \ref{jbdevdouycdit} imply Theorem \ref{sibevhz2}.

\begin{proof} 
Let $(\mathcal{F}=(V,\mathcal{E}\cup \mathcal{A}),f,g,k,\ell,\ell')$ be an instance of Theorem \ref{sibevhz2} that satisfies \eqref{bdkzjdju}, \eqref{gkvl} and \eqref{jvljh1hyp2}. Since \eqref{jvljh1hyp2} holds, by Lemma \ref{jbdevdouycdit}(c), \eqref{jbipuib1} and \eqref{jbipuib2} hold. Since \eqref{bdkzjdju} and \eqref{gkvl} also hold, by Lemma \ref{jbdevdouycdit}(b), the polyhedron $T,$ defined in Lemma \ref{jbdevdouycdit}, is not empty. We have seen in the proof of Lemma \ref{jbdevdouycdit}(b) that $T$ is the intersection of two generalized polymatroids $Q(p, b)$ and $Q(0,r_{{\sf M}_{\mathcal{F}}^k}).$ Then, by Theorem \ref{gpmip}.\ref{g}(ii),  $T$ contains an integer point $x.$ By Lemma \ref{jbdevdouycdit}(b), $x$ is the characteristic vector of the dyperedge set of an $(f,g)$-bounded $k$-regular $(\ell,\ell')$-limited packing of hyperarborescences in an orientation $\vec{F}=(V,\vec{\mathcal{E}}\cup \mathcal{A})$ of $\mathcal{F}$. By replacing the dyperedges in $\vec{\mathcal{E}}$ by the underlying hyperedges in $\mathcal{E}$, we obtain the required packing.
\end{proof}

\section{Acknowledgements}

I  thank Csaba Kir\'aly and Pierre Hoppenot for their very careful reading of the paper.


\begin{thebibliography}{99}

\bibitem{BF2} 
{\scshape K. B\'erczi, A. Frank,} 
\newblock{Variations for Lov\'asz' submodular ideas,} 
\newblock  {\itshape  in: M. Gr\"otschel et. al. eds.  Building Bridges,} Springer, (2008) 137--164.

\bibitem{BF3} 
{\scshape K. B\'erczi, A. Frank,}
\newblock{Supermodularity in Unweighted Graph Optimization I: Branchings and Matchings,}  
\newblock  {\itshape  Math. Oper. Res.} {\bfseries 43}(3) (2018) 726--753.

\bibitem{cai1} 
{\scshape M. C. Cai},
\newblock Arc-disjoint arborescences of digraphs,
\newblock  {\itshape  J. Graph Theory} {\bfseries 7} (1983) 235--240.

\bibitem{Sz}  
{\scshape O. Durand de Gevigney, V. H. Nguyen, Z. Szigeti}, 
\newblock {Matroid-Based Packing of Arborescences}, 
\newblock  {\itshape  SIAM J. Discret. Math.} {\bfseries 27} (1) (2013)  567--574.

\bibitem{Egy} 
{\scshape J. Edmonds}, 
\newblock Edge-disjoint branchings, 
\newblock  {\itshape  Combinatorial Algorithms, B. Rustin ed., Academic Press,} New York, (1973)  91--96.

\bibitem{FKLSzT}  
{\scshape Q. Fortier, Cs. Kir\'aly, M. L\'eonard, Z. Szigeti, A. Talon}, 
\newblock {Old and new results on packing arborescences}, 
\newblock  {\itshape  Discret. Appl. Math.} {\bfseries 242} (2018) 26--33. 

\bibitem{FA} 
{\scshape A. Frank}, 
\newblock On disjoint trees and arborescences,
 \newblock    {\itshape  in: L. Lov\'asz et. al. eds.   Algebraic Methods in Graph Theory, 25, Colloquia Mathematica Soc. J. Bolyai,} North-Holland, (1978) 59--169.
 
 \bibitem {fgp} 
 {\scshape A. Frank,}
\newblock Generalized polymatroids, 
 \newblock  {\itshape in: A. Hajnal et. al. eds. Finite and infinite sets,} North-Holland, Amsterdam-New York (1984) 285--294.

 \bibitem{book} 
 {\scshape A. Frank}, 
\newblock Connections in Combinatorial Optimization, 
\newblock  {\itshape  Oxford University Press}, 2011.

\bibitem {fkiki} 
{\scshape A. Frank, T. Kir\'aly,  Z. Kir\'aly},
\newblock On the orientation of graphs and hypergraphs, 
\newblock  {\itshape Discret. Appl. Math.} {\bfseries 131}(2) (2003) 385--400.

\bibitem{GY}
{\scshape H. Gao, D. Yang}, 
\newblock Packing of maximal independent mixed arborescences, 
\newblock  {\itshape Discret. Appl. Math. } {\bfseries 289} (2021) 313--319.

\bibitem{GY2} 
{\scshape H. Gao, D. Yang}, 
\newblock Packing of spanning mixed arborescences, 
\newblock  {\itshape J. Graph Theory}, {\bfseries 98} (2) (2021) 367--377.

\bibitem{hassin} 
{\scshape R. Hassin}, 
\newblock Minimum cost flow with set-constraints, 
\newblock  {\itshape Networks}  {\bfseries 12} (1)  (1982) 1--21.

\bibitem{HSz} 
{\scshape F. H\"orsch, Z. Szigeti},
\newblock Packing of mixed hyperarborescences with flexible roots via matroid intersection, 
\newblock  {\itshape Electronic Journal of Combinatorics}, {\bfseries 28} (3) (2021) P3.29.

\bibitem{HSz2} 
{\scshape F. H\"orsch, Z. Szigeti},
\newblock Reachability in arborescence packings,
\newblock  {\itshape  Discret. Appl. Math.} {\bfseries 320} (2022) 170--183.

\bibitem{japan}  
{\scshape N. Kamiyama, N. Katoh, A. Takizawa}, 
\newblock Arc-disjoint in-trees in directed graphs,
\newblock  {\itshape   Comb.} {\bfseries 29} (2009) 197--214.

\bibitem {cskir} 
{\scshape Cs. Kir\'aly}, 
\newblock On maximal independent arborescence packing,
\newblock  {\itshape SIAM J. Discret. Math.} {\bfseries 30}(4) (2016) 2107--2114.

\bibitem{l} 
{\scshape M. Lorea,}  
\newblock Hypergraphes et matroides, 
\newblock  {\itshape Cahiers Centre Etudes Rech. Oper.} {\bfseries 17} (1975) 289-291.

\bibitem{MT}
{\scshape T. Matsuoka, S. Tanigawa},
\newblock On Reachability Mixed Arborescences Packing,
\newblock  {\itshape Discret. Optim. } {\bfseries 32} (2019) 1--10.  

\bibitem{hujaproc}
{\scshape Z. Szigeti,} 
\newblock Packing mixed hyperarborescences, 
\newblock Proceedings of the 12th Hungarian-Japanese Symposium on Discrete Mathematics and Its Applications, 2023, Budapest, Hungary

\bibitem{survey}
{\scshape Z. Szigeti,} 
\newblock A survey on packing arborescences, in preparation

\end{thebibliography}
\end{document}